\documentclass[11pt]{amsart}

\usepackage{tikz}

\usepackage{mathpazo}
\usepackage{amsmath,amssymb,amsthm}
\usepackage[latin1]{inputenc}
\usepackage{version,tabularx,multicol}
\usepackage{graphicx,float,psfrag}
\usepackage{stmaryrd}
 \usepackage{color}
\usepackage[pdftex]{hyperref}
\usepackage[numbers,sort&compress]{natbib}

\newcommand{\R}{\mathbb{R}}

\newcommand{\E}{\mathbb{E}}

\def\build#1_#2^#3{\mathrel{ \mathop{\kern 0pt#1}\limits_{#2}^{#3}}}

\def\d{\mathrm{d}}

\def\cN{{\mathcal  N}}

\def\d{\mathrm{d}}

\renewcommand{\P}{\mathbb{P}}

\newcommand{\ind}{\mathbf{1}}

\theoremstyle{plain}
\newtheorem{thm}{Theorem}[section]
\newtheorem{lmm}[thm]{Lemma}

\newtheorem{crl}[thm]{Corollary}

\title{Large deviations for the maximum of a reducible two-type branching Brownian motion}

\author{Hui He}
\address{Hui He,
Laboratory of Mathematics and Complex Systems, School of Mathematical Sciences, Beijing Normal University, P.R. China}
\email{hehui@bnu.edu.cn}

\date{\today}

\begin{document}

%\selectlanguage{english}

%\maketitle

\begin{abstract}
We consider a two-type reducible branching Brownian motion, defined as a particle
system on the real line in which particles of two types move according to independent
Brownian motions and create offspring at a constant rate. Particles of type $1$ can give
birth to particles of types $1$ and $2$, but particles of type $2$ only give births to descendants
of type $2$. Under some specific conditions, Belloum and Mallein in \cite{BeMa21} showed that  the maximum position $M_t$ of all particles alive at time $t$, suitably centered   by a deterministic function $m_t$, converge weakly. In this work, we are interested in the decay rate of the following upper large deviation probability, as $t\rightarrow\infty$,
\[
\P(M_t\geq \theta m_t),\quad \theta>1.
\]
We shall show that the decay rate function exhibits phase transitions depending on  certain relations between $\theta$, the variance of the underlying Brownian motion and the branching rate. 
\end{abstract}



\maketitle

\section{Introduction and main results}\label{sec:introduction}

Branching Brownian motion (BBM) is a spatial branching process of great interests in recent years. On the one hand, its connection with the partial differential Fisher-Kolmogorov-Petrovsky-Piskunov (F-KPP) equation brings out many studies on both probabilistic and analytic sides, see for instance, \cite{ABBS13, McK75, Bra83, LS87, ABK13}. On the other hand, it is the fundamental model to understand the BBM-universality class which includes the 2-dim Gaussian free field \cite{BL16, BDZ16} and 2-dim cover times \cite{DPRZ04}, etc.

\smallskip

A one-dimensional  standard binary branching Brownian motion is a continuous-time particle system on the real line which can be constructed as follows. It starts with one individual located at the origin at time $0$ that moves according to a Brownian motion with variance $\sigma^2>0$. After an independent exponential time of parameter $\beta$, the initial particle dies and gives birth to 2 children that start at the position their parent occupied at its death. These $2$ children then move according to independent Brownian motions with variance $\sigma^2$ and give birth independently to their own children at rate $\beta$. The particle system keeps evolving in this fashion for all time.
\smallskip

In this work, we shall study the two-type reducible branching Brownian motion. This is a particle system on the real line in which particles possess a type in addition with their position. Particles of type $1$ move according to Brownian motions with variance $\sigma^2$ and branch at rate $\beta$ into two children of type $1$. Additionally, they give birth to particles of type $2$ at rate $\alpha$. Particles of type $2$ move according to Brownian motions with variance $\sigma_2^2$ and branch at rate $\beta_2$, but cannot give birth to descendants of type $1$. Up to a dilation of time and space, in this work we shall always assume that $\beta_2=\sigma_2=1$.  This model has been studied extensively by many authors; see \cite{Bi12, Ho14,Ho15, MR23}.

With abuse of notation, for all $t \geq 0$, we denote by $\cN_t$ the collection of individuals alive at time $t$ for both the standard BBM and the two-type reducible BBM. And denote by $\cN_t^1$ and $\cN_t^2$ the set of particles of type 1 and type 2 at time $t$, respectively.  For any $u\in \cN(t)$ and $s \leq t$, let $X^u_s$ denote the position at time $s$ of the individual $u$ or its ancestor alive at that time. The maximum of the branching Brownian motion at time $t$ is defined as
\[\displaystyle
  M_t:=\max\{X^u_t: u\in \cN(t)\}.
\]

To distinguish the two-type reducible BBM and the standard BBM,  we shall regard the standard BBM as $(\sigma^2,\beta)$-BBM to emphasize the variance of the Brownian motion and the branching rate. We also denote by $\P_{\sigma^2, \beta}$ its law and write $\bar{X}^u_s$ for the position at time $s$ of the individual $u$ of a $(\sigma^2, \beta)$-BBM.
For $(\sigma^2, \beta)$-BBM,  the asymptotic behaviour of
\[
\bar{M}_t:=\max\{\bar{X}_t^u: u\in \cN(t)\},
\]
as $t \to \infty$, has been subjected to intense study, partly due to its link with the F-KPP reaction-diffusion equation, defined as
\begin{equation}
  \label{eqn:deffkpp}
\partial_t u(t, x)=\frac{1}{2}\partial^2_xu(t, x)+(1-u(t,x))-(1-u(t,x))^2,\quad t\geq0,\, x\in {\mathbb R}.
\end{equation}
Precisely, McKean \cite{McK75} showed that 
\[\P_{\sigma^2, 1}(\bar{M}_t \ge \sigma x)=\bar{u}(t,  x),\] 
where the function $(t,x) \mapsto \bar{u}(t, x)$ is the unique solution of \eqref{eqn:deffkpp} with initial condition $\bar{u}(0, x) = \ind_{\{x < 0\}}$. With the help of \eqref{eqn:deffkpp}, Bramson \cite{Bra83} proved that for a $(\sigma^2, \beta)$-BBM,
\[
\bar{M}_t-\sqrt{2 \sigma^2 \beta} t+\frac{3}{2 \sqrt{2 \beta / \sigma^2}} \ln t,
\]
converges weakly, as $t\rightarrow\infty$. In particular, uniformly in $x \in \R$, as $t\rightarrow\infty$,
\begin{eqnarray}\label{Cvgmax}
1-\bar{u}(t, \bar{m}_t+x)=\P_{\sigma^2, 1}\left(\frac{\bar{M}_t}{\sigma}-\sqrt{2}t+\frac{3}{2\sqrt{2}}\ln t\leq x \right)\longrightarrow \bar{\omega}(x),
\end{eqnarray}
where $ \bar{\omega}$ is the travelling wave solution of the F-KPP equation at speed $\sqrt{2}$, which satisfies
\[
  \frac{1}{2}  \bar{\omega}'' + \sqrt{2}  \bar{\omega}' -  \bar{\omega}(1 -  \bar{\omega}) = 0.
\]
 Lalley and Sellke \cite{LS87}  characterized  $\bar{\omega}$ in terms of a random shift of the Gumbel distribution by introducing the so-called derivative martingale.

Parallel results of \eqref{Cvgmax} for the two-type reducible BBM were obtained in \cite{BeMa21}. To state their results, we need to decompose the state space $(\beta, \sigma^2)\in {\mathbb R}_+^2$ into three regions (see Figure 1 below for the phase diagram):

\begin{align*}
\mathcal{C}_\text{\uppercase\expandafter{\romannumeral1}} & =\left\{\left(\beta, \sigma^2\right): \sigma^2>\frac{\mathbb{1}_{\{\beta \leq 1\}}}{\beta}+\mathbb{1}_{\{\beta>1\}} \frac{\beta}{2 \beta-1}\right\}, \\
\mathcal{C}_{\text{\uppercase\expandafter{\romannumeral2}}} & =\left\{\left(\beta, \sigma^2\right): \sigma^2<\frac{\mathbb{1}_{\{\beta \leq 1\}}}{\beta}+\mathbb{1}_{\{\beta>1\}}(2-\beta)\right\}, \\
\mathcal{C}_{\text{\uppercase\expandafter{\romannumeral3}}} & =\left\{\left(\beta, \sigma^2\right): \sigma^2+\beta>2 \text { and } \sigma^2<\frac{\beta}{2 \beta-1}\right\} .
\end{align*}

\begin{figure}
\centering
\includegraphics[scale=0.28]{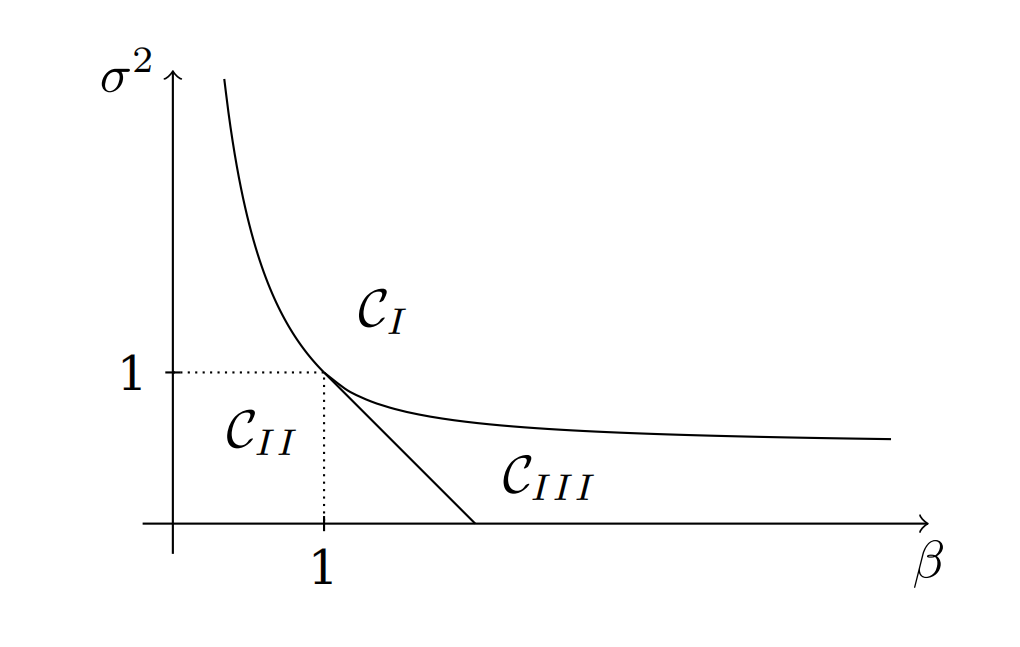}
\caption{Phase diagram of the two-type reducible BBM}
\end{figure}

\noindent From Belloum and Mallein \cite{BeMa21}, we have:
\[
M_t-m_t
\]
converges weakly, as $t\rightarrow\infty$, where
\begin{align*}
m_t=\begin{cases}
\sqrt{2 \sigma^2 \beta} t-\frac{3}{2 \sqrt{2 \beta / \sigma^2}} \ln t,\quad & (\beta, \sigma^2)\in \mathcal{C}_{\text{\uppercase\expandafter{\romannumeral1}}},\cr
\sqrt{2} t-\frac{3}{2 \sqrt{2}} \ln t,\quad
& (\beta, \sigma^2)\in \mathcal{C}_{\text{\uppercase\expandafter{\romannumeral2}}},\cr
\frac{\sigma^2-\beta}{\sqrt{2\left(1-\sigma^2\right)(\beta-1)}} t,\quad 
&(\beta, \sigma^2)\in \mathcal{C}_{\text{\uppercase\expandafter{\romannumeral3}}}.
\end{cases}
\end{align*}
From above, we see that if $\left(\beta, \sigma^2\right) \in \mathcal{C}_I$, the speed of the two-type reducible BBM is $\sqrt{2 \beta \sigma^2}$, which is the same as particle s of type $1$ alone, ignoring births of particles of type $2$.  Conversely, if $\left(\beta, \sigma^2\right) \in \mathcal{C}_{\text{\uppercase\expandafter{\romannumeral2}}}$, then the speed of the process is $\sqrt{2}$, equal to the speed of a single BBM of particles of type $2$.  Finally, if $\left(\beta, \sigma^2\right) \in \mathcal{C}_{\text{\uppercase\expandafter{\romannumeral3}}}$, the speed of the process is a mixture of the long-time asymptotic of the processes of particles of type $1$ and $2$,  which is larger than $\max \left(\sqrt{2}, \sqrt{2 \beta \sigma^2}\right)$.

We remark here that before the work of Belloum and Mallein \cite{BeMa21}, Biggins' work \cite{Bi12} was the first to consider the maximum of multi-type reducible branching random walks. After the work of Belloum and Mallein \cite{BeMa21}, as $\mathcal{C}_I, \mathcal{C}_{II}, \mathcal{C}_{III}$ does not cover all the possible values for $(\beta, \sigma^2)\in (0,\infty)^2$, Belloum \cite{Belloum2022} studied the case of $\beta=\sigma^2=1$. Later, Ma and Ren \cite{MR23} studied the other cases and completed the phase diagram.

 In this article, we study the large deviation probabilities for the maximum of a two-type reducible BBM. That is the convergence rate of the following probability 
 \[
 \P(M_t> \theta m_t),\quad \theta>1,
 \]
as $t\rightarrow\infty$. We divide the main result of our article into three theorems, one for each area the pair $\left(\beta, \sigma^2\right)$ belongs to.

%\section{Main results}

\begin{thm}\label{ThmcaseI}

Assume $\left(\beta, \sigma^2\right) \in \mathcal{C}_\text{\uppercase\expandafter{\romannumeral1}}$. Then
the following limit exists for any $\theta>1$,

$$\lim _{ t\to \infty}\frac{1}{t}\ln \P\left(M_t\geq \sqrt{2 \beta \sigma^2\theta^2 }\cdot  t \right)=A,$$

\noindent where if    $\sigma ^2\geq 1$, then

$$ A=\beta\left(1-\theta^2\right),$$

\noindent and if $\beta > 1$ and $\sigma ^2< 1$, then we have

$$A=\begin{cases}\beta\left(1-\theta^2\right), & \theta^2\leq \frac{(\beta -1)\sigma ^2}{\left(1-\sigma^2\right)  \beta}, \\
1-\theta^2 \beta \sigma^2, & \theta^2\geq \frac{(\beta-1) }{\left(1-\sigma^2\right) \sigma ^2\beta} ,\\
\frac{-2 \sqrt{\beta(\beta-1)\left(1-\sigma^2\right)} \sigma \theta+\beta-\sigma^2}{1-\sigma^2}, &  \frac{(\beta -1)\sigma ^2}{\left(1-\sigma^2\right)  \beta}< \theta^2< \frac{(\beta-1) }{\left(1-\sigma^2\right) \sigma ^2\beta}.\end{cases}$$
\end{thm}

\begin{thm}\label{ThmcaseII}

  Assume $\left(\beta, \sigma^2\right) \in \mathcal{C}_\text{\uppercase\expandafter{\romannumeral2}}$. Then the following limit exists for any $\theta>1$,

  $$\lim _{ t\to \infty}\frac{1}{t}\ln \P\left(M_t\geq \sqrt{2}\theta t \right)=:A,$$

  \noindent where if  $\sigma ^2 \leq   1$, we have

  $$ A=1-\theta^2 ,$$

  \noindent and if $\beta \leq 1,\, \sigma ^2 > 1$, we have

  $$A=\begin{cases} 1- \theta^2, & \theta^2\leq \frac{(1-\beta) }{\left(\sigma^2-1\right) },
  \cr
    \beta -\frac{\theta ^2}{\sigma ^2,} & \theta^2\geq \frac{(1-\beta) }{\left(\sigma^2-1\right) }\sigma ^4, 
    \cr
    \frac{2 \sqrt{(1-\beta)\left(\sigma^2-1\right)}  \theta+\beta-\sigma^2}{1-\sigma^2}, &  \frac{(1-\beta) }{\left(\sigma^2-1\right) } <  \theta^2 <   \frac{(1-\beta) }{\left(\sigma^2-1\right) }\sigma ^4.\end{cases}$$
  \end{thm}

\begin{thm}\label{ThmcaseIII}
Assume $\left(\beta, \sigma^2\right) \in \mathcal{C}_\text{\uppercase\expandafter{\romannumeral3}}$. Then for any $\theta>1$, we have
 $$
\lim _{ t\to \infty}\frac{1}{t}\ln \P\left(M_t\geq \xi \theta t \right)=\begin{cases}
1-\frac{\xi^2\theta^2}{2}, & \theta^2\geq\frac{2(\beta-1)}{1-\sigma^2}\frac{1}{\xi^2},\\
  \beta-\frac{\xi^2\theta^2}{2 \sigma^2}, & \theta^2\leq \frac{2(\beta-1)}{1-\sigma^2}\frac{\sigma^4}{\xi^2},
  \cr
  \frac{\sqrt{2}(1-\beta)-\theta^2\left(\sigma^2-\beta\right)^2-1}{\sqrt{2}\left(\sigma^2-1\right)}+1, & 
  \frac{2(\beta-1)}{1-\sigma^2}\frac{\sigma^4}{\xi^2} < \theta^2< \frac{2(\beta-1)}{1-\sigma^2}\frac{1}{\xi^2}, \end{cases}
$$
where $\xi =\frac{\sigma ^2-\beta }{\sqrt{2(1-\sigma ^2)(\beta -1)} }. $
\end{thm}

Let us mention here that in the context of standard $(1,1)$-BBM (i.e., $\sigma^2=1, \beta=1$), one would have: for any $\theta>1$ and $x\in {\mathbb R}$,
 \begin{align}\label{cvgsbbm}
 \lim_{t\rightarrow\infty}\sqrt{t}e^{(\theta^2-1)t}\P_{1,1}(\bar{M}_t\geq\sqrt{2}\theta t+x)&
 =\lim_{t\rightarrow\infty}\sqrt{t}e^{(\theta^2-1)t}\bar{u}(t, \sqrt{2}\theta t+x)
 \cr& =e^{\sqrt{2}x}C(\theta, \bar{u}),\end{align}
where $C(\theta, \bar{u})$ is a constant depending only on $\theta$ and $\bar{u}$; see \cite{CR88}, \cite{ABK11} and \cite{BBCM22}. In this work, we only obtain the exponential decay rates for the large deviation probabilities for the maximum of a two-type reducible BBM. It is possible to obtain a form as \eqref{cvgsbbm} with the help of the so-called truncated second moment method and the F-KPP equation. Our method may also apply to study the cases of  $(\beta, \sigma^2)\in \partial\mathcal{C}_\text{\uppercase\expandafter{\romannumeral1}} \cap \partial\mathcal{C}_\text{\uppercase\expandafter{\romannumeral2}}$ or $(\beta, \sigma^2)\in \partial\mathcal{C}_\text{\uppercase\expandafter{\romannumeral2}} \cap \partial\mathcal{C}_\text{\uppercase\expandafter{\romannumeral3}}$ or  $(\beta, \sigma^2)\in \partial\mathcal{C}_\text{\uppercase\expandafter{\romannumeral1}} \cap \partial\mathcal{C}_\text{\uppercase\expandafter{\romannumeral3}}$; see \cite{MR23} for related results.

In the next section, we will present some technical lemmas borrowed from the existing literature.  The proofs for Theorems \ref{ThmcaseI}, \ref{ThmcaseII} and \ref{ThmcaseIII} will be given in Sections \ref{SecCaseI}, \ref{SecCaseII} and \ref{SecCaseIII}, respectively.

\section{Technical lemmas}\label{sectl}
\begin{lmm}(\cite{GKS18}, \label{lemGSK}Lemma 2.1)
Let $X$ be a centered Gaussian random variable with variance $\sigma^2>0$. Then
\begin{eqnarray}\label{BMasy}\P(X>a)=(2\pi)^{-1/2}(\sigma/a)
e^{-\frac{a^2}{2\sigma^2}}\left(1+O(\sigma^2/a^2)\right),\quad \text{as } a/\sigma\rightarrow\infty,
\end{eqnarray}
with the r.h.s. above without error term being an upper bound valid for any $a>0.$
\end{lmm}

%\noindent Furthermore,  for $(\sigma^2, \beta)$-BBM and $0<x<1$, $\P_{\sigma^2,\beta}-a.s.$,
%\begin{eqnarray}\label{levelset}
%\frac{1}{t}\log \cN_{t}\left(\sqrt{2\beta}x\sigma t)\right)\overset{t\rightarrow\infty}{\longrightarrow} (1-x^2)\beta,
%\end{eqnarray}
%where  $\cN_t(x)$ denote the number of particles, in the $(\sigma^2, \beta)$-BBM, alive at time $t$ and positioned in $[tx, \infty)$; see \cite{Bi95} {\red and \cite{GKS18}.}

 If $u \in \mathcal{N}_t^2$, we denote by $T(u)$ the time at which the oldest ancestor of type 2 of $u$ was born. The set $\mathfrak{B}$ of particles of type $2$ that are
born from a particle of type $1$, which can be defined as 
\[\mathfrak{B}:=\{u\in \cup_{t\geq 0} \cN^2_t: T(u)=b_u\},
\]
where $b_u$ is the birth time of $u$.
We observe that $\mathfrak{B}$ can be thought of as a Poisson point process with random intensity.
\begin{lmm}(\cite{BeMa21}, \label{lemPPP}Lemma 4.2)
Set ${\mathcal F}_1 =\sigma(X_t^u, u\in \cN_t^1,t\geq0)$. The point measure $\sum_{u\in \mathfrak{B}} \delta_{(X_s^u, s\leq T(u))}$
conditionally on ${\mathcal F}_1$  is a Poisson point process with intensity $\alpha {\rm d}t \times\sum_{u\in {\mathcal N}_t^1} \delta_{(X_s^u,\, s\leq t)}$.
\end{lmm}
The results in the following proposition are variations of the so-called many-to-one lemma.
\begin{lmm}\label{lemMany-one}  For all measurable non-negative functionals $f$ and $g$, we have
\begin{align*}
\E
\left[\sum_{u\in \cN_t^1} f(X_s^u, s \leq  t)\right]
=
 e^{\beta t}\E\left[f(\sigma B_{s} , s\leq t)\right]
\end{align*}
and
\begin{align*}
&\E
\left[\sum_{u\in \cN_t^2} g((X_s^u, s \leq  t), T(u))\right]\cr
&\quad =
\alpha\int_0^t e^{\beta s+(t-s)}\E\left[g((\sigma B_{u\wedge s} +(B_u-B_{u\wedge s}), u\leq t),s)\right]{\rm d} s.
\end{align*}
\end{lmm}
\begin{proof}
The first result is just the usual many-to-one lemma. The second can be found in Proposition 4.1 in \cite{BeMa21}. 
\end{proof}
\section{Proof of Theorem \ref{ThmcaseI}}\label{SecCaseI}
In this section, we consider the case that
\[(\beta, \sigma^2)\in\mathcal{C}_\text{\uppercase\expandafter{\romannumeral1}} =\left\{\left(\beta, \sigma^2\right): \sigma^2>\frac{\mathbb{1}_{\{\beta \leq 1\}}}{\beta}+ \frac{\beta\mathbb{1}_{\{\beta>1\}}}{2 \beta-1}\right\}.
\]

 Before going into the details, we would like to give a rough explanation about the general strategy to all proofs. Our method is quite simple and direct. For the upper bound, we just use the first moment argument. Then the task is reduced to understand the birth time $T(u)$ of the particle whose position at time $t$ is larger than $\sqrt{2 \beta \sigma^2\theta^2 }\cdot  t.$ For the lower bound, we may think that oldest ancestor of type $2$ of the maximal particle at time $t$ was born at some time $u_0 t$ in a position around $x_0 t$ for some $u_0$ and $x_0$.
So the limit $A$ is consist of two parts: the effort of some type $1$ particles to achieve position $x_0 t$ at time $u_0 t$ and the effort that some  type $2$ particles to achieve position $\sqrt{2 \beta \sigma^2\theta^2 }\cdot  t-x_0 t$ at time $t-u_0 t$. The same strategy also apply well to the other two cases.

\subsection{Upper bound}
\begin{proof}

First,
\begin{align}
 & \P\left(M_t \geqslant \sqrt{2 \beta \sigma^2\theta^2 }\cdot  t \right) \cr
= & \P\left(\exists\, u \in \cN_t, X_t^u \geqslant \sqrt{2 \beta \sigma^2\theta^2 }\cdot  t \right) \cr
\leqslant & \P\left(\exists\, u \in \cN_t^1, X_t^u \geqslant \sqrt{2 \beta \sigma^2\theta^2 }\cdot  t \right)+\P\left(\exists\, u \in \cN_t^2, X_t^u \geqslant \sqrt{2 \beta \sigma^2\theta^2 }\cdot  t \right) \cr
 =:&  \text{\uppercase\expandafter{\romannumeral1}}+\text{\uppercase\expandafter{\romannumeral2}}.
\end{align}

\noindent Then by Lemma \ref{lemMany-one} and \eqref{BMasy}, we have
\begin{align*}
  \text{\uppercase\expandafter{\romannumeral1}} & \leq e^{\beta t}  \P\left(\sigma B_t \geqslant\sqrt{2 \beta \sigma^2\theta^2 }\cdot  t \right)
  \cr 
  &\leq O(1) \cdot \frac{1}{\sqrt{t}} \cdot e^{\beta t} \cdot e^{-\frac{2 \beta \theta^2 t^2}{2 t}} \\
& =\frac{O(1)}{\sqrt{t}} \cdot e^{\beta\left(1-\theta^2\right) t}.  
\end{align*}

\noindent On the other hand, set\[\mathfrak{B}_t= \{u\in \mathfrak{B}: T(u)<t\}.\]
According to construction of our model, we have
\begin{align*}%\label{C1L01}
    & \P\left(\exists\, u \in \cN_t^2: X_t^u \geqslant \sqrt{2 \beta \sigma^2 \theta^2} \cdot   t\right) \cr
  & =  \P\left(\exists u\, \in \mathfrak{B}_t : X_{T(u)}^u+M_{t-T(u)}^u \geqslant \sqrt{2 \beta \sigma^2 \theta^2} \cdot  t\right) \cr
 &=  1-\P\left(\forall\, u \in \mathfrak{B}_t : X_{T(u)}^u+M_{t-T(u)}^u < \sqrt{2 \beta \sigma^2 \theta^2} \cdot  t\right) \cr
&=  1-\E\left[\prod_{u \in \mathfrak{B}_t } \varphi_t\left(X_{T(u)}^u, T(u)\right)\right] \cr
&=  1-\E\left[\exp \left(-\sum_{u \in \mathfrak{B}_t }-\ln \varphi_t\left(X_{T(u)}^u, T(u)\right)\right)\right],
\end{align*}

\noindent where \begin{align*}
   1- \varphi_t(x, s)  = \bar{u}\left(t-s,\, \sqrt{2 \beta \sigma^2 \theta^2}  \cdot  t-x\right)
    = \P_{1,1}\left(\bar{M}_{t-s} \geq \sqrt{2\beta  \sigma^2 \theta^2 } \cdot t-x\right). 
\end{align*} 
and  $\left\{M^u_{t-T(u)}: u\in \mathfrak{B}_t\right\}$ are independent random variables such that $M^u_{t-s}$ has the same law as $\bar{M}_{t-s}$ under $\P_{1, 1}$. Applying Lemmas \ref{lemPPP}, we have
\begin{align}\label{C1U02}
    & \P\left(\exists\, u \in \cN_t^2: X_t^u \geqslant \sqrt{2 \beta \sigma^2 \theta^2} \cdot   t\right) \cr
  & =   1-\E\left[\exp\left\{-\alpha\int_0^t \sum_{u\in \cN_s^1}  \bar{u}\left(t-s, \sqrt{2 \beta\sigma^2 \theta^2} \cdot  t-X_s^u\right){\rm d} s\right\}\right].
  \end{align}
  With Lemma \ref{lemMany-one} in hand, applying elementary inequality $1-e^{-x}\leq x$ for $x\geq0$, we have
\begin{align}\label{C1U02a}
 \text{\uppercase\expandafter{\romannumeral2}}& \leq  \E\left[\alpha\int_0^t \sum_{u\in \cN_s^1}  \bar{u}\left(t-s, \sqrt{2 \beta\sigma^2 \theta^2} \cdot  t-X_s^u\right){\rm d} s\right]\cr
  &=
  \alpha\int_0^t{\rm d} s  \int_{-\infty}^{+\infty} e^{\beta s} \,  \bar{u}\left(t-s, \sqrt{2 \beta\sigma^2 \theta^2} \cdot  t-x\right) \frac{1}{\sqrt{2\pi\sigma^2 s}}e^{-\frac{x^2}{2\sigma^2s}}{\rm d}x\cr
  &=
  \alpha t^{3/2}\int_0^1{\rm d} s  \int_{-\infty}^{+\infty} e^{\beta s t} \,  \bar{u}\left((1-s)t, \sqrt{2 \beta\sigma^2 \theta^2} \cdot  t-xt\right) \frac{1}{\sqrt{2\pi\sigma^2 s }}e^{-\frac{x^2 t}{2\sigma^2s}}{\rm d}x.
\end{align}
By \eqref{cvgsbbm}, for any $0<s<1$ and $x<\sqrt{2 \beta\sigma^2 \theta^2} -\sqrt{2}(1-s)$, we have, as $t\rightarrow\infty$, 
\[
\bar{u}\left((1-s)t, \sqrt{2 \beta\sigma^2 \theta^2} \cdot  t-xt\right) \leq 
O(1)\frac{1}{\sqrt{t}} \exp{\left\{\left(1-s-\frac{(\sqrt{2\beta\sigma^2\theta^2}-x)^2}{2(1-s)}\right)t\right\}}.
\]
Thus
\[
I_2\leq O(1)t\int_0^1{\rm d} u  \int_{-\infty}^{+\infty} e^{f_1^*(u, x)t}\frac{1}{\sqrt{u }}{\rm d}x,
\]
where for $0<u<1,\, x\in {\mathbb R},$
\[
f^*_1(u, x)= \beta u-\frac{x^2}{2\sigma^2 u}+\left(1-u-\frac{(\sqrt{2\beta\sigma^2\theta^2}-x)^2}{2(1-u)}\right) {\bf 1}_{\{x<\sqrt{2 \beta\sigma^2 \theta^2} -\sqrt{2}(1-u) \}},
\]
By studying function $x\mapsto f^*_1(u, x)$, one can see that
\begin{eqnarray}\label{fu}
\max_{x\in {\mathbb R}}f^*_1(u, x)&=&f^*_1(u, x^*(u))=:f_1(u) \cr
                        &=&\begin{cases}(\beta-1)u+1-\frac{\beta\sigma^2\theta^2}{1-u+\sigma^2 u},\quad &\text{if }u(\sigma^2-1)<\sqrt{\beta\sigma^2 \theta^2}-1,\cr
                         \beta u-\frac{(\sqrt{\beta\sigma^2 \theta^2} -(1-u))^2}{\sigma^2 u},\quad &\text{if }u(\sigma^2-1)\geq \sqrt{\beta\sigma^2 \theta^2}-1,
                         \end{cases}                        
\end{eqnarray}
with
\[x^*(u)= \begin{cases} \frac{ \sqrt{2 \beta} \theta \sigma^3 u}{1-u+\sigma^2 u},\quad &\text{if }u(\sigma^2-1)<\sqrt{\beta\sigma^2 \theta^2}-1,\cr
                         \sqrt{2 \beta\sigma^2 \theta^2} -\sqrt{2}(1-u)
                         ,\quad &\text{if }u(\sigma^2-1)\geq \sqrt{\beta\sigma^2 \theta^2}-1,
                         \end{cases}    \] 
where $u(\sigma^2-1)<\sqrt{\beta\sigma^2 \theta^2}-1 $ is equivalent to $\frac{\sqrt{2 \beta} \theta \sigma^3 u}{1-u+\sigma^2 u}<\sqrt{2 \beta\sigma^2 \theta^2}-\sqrt{2}(1-u).$
\noindent Furthermore, 
\[
f'_1(u) =\begin{cases}\beta-1+\frac{\beta\sigma^2\theta^2(\sigma^2-1)}{(1-u+\sigma^2 u)^2},\quad &\text{if }u(\sigma^2-1)<\sqrt{\beta\sigma^2 \theta^2}-1,\cr
\frac{(\beta\sigma^2-1)u^2+(\sqrt{\beta\sigma^2\theta^2}-1)^2}{\sigma^2u^2},\quad &\text{if }u(\sigma^2-1)\geq \sqrt{\beta\sigma^2 \theta^2}-1.
                         \end{cases}                     
\]
%\beta -\frac{2(\sqrt{\beta\sigma^2 \theta^2} -(1-u))}{\sigma^2 u}+\frac{(\sqrt{\beta\sigma^2 \theta^2} -(1-u))^2}{\sigma^2 u^2}
 % and
%\[
%f''_1(u) =\begin{cases}-\frac{2\beta\sigma^2\theta^2(\sigma^2-1)^2}{(1-u+\sigma^2 u)^3},\quad &\text{if }u(\sigma^2-1)<\sqrt{\beta\sigma^2 \theta^2}-1,\cr
%\frac{u^2-2(\sqrt{\beta\sigma^2 \theta^2} -1)^2}{\sigma^2 u^3},\quad &\text{if }u(\sigma^2-1)\geq \sqrt{\beta\sigma^2 \theta^2}-1,
%                         \end{cases}                     
%\]  
 Also notice that for $\left(\beta, \sigma^2\right) \in \mathcal{C}_\text{\uppercase\expandafter{\romannumeral1}}$, we always have \[\beta\sigma^2\geq \begin{cases}
 1,\quad &\beta\leq 1,\cr
 \frac{\beta^2}{2\beta-1}\geq1,\quad &\beta >1. 
 \end{cases}\] Then we have two cases.

 {\bf Case I.} If $\left(\beta, \sigma^2\right) \in \mathcal{C}_\text{\uppercase\expandafter{\romannumeral1}}$ and $\sigma^2\geq 1$, then
$f_1^{\prime}(u)\geq 0$ for $0<u<1$.  Thus $\max_{0\leq u\leq 1}f_1(u)=f_1(1)$ and hence
 \[
 \max_{0\leq u\leq 1,\, x\in \mathbb R}f^*_1(u,x)=f_1\left(1\right)=\beta(1-\theta^2).
 \]
 Hence
\begin{equation}\label{da1}
   \limsup_{t \rightarrow \infty} \frac{1}{t} \ln \P\left(M_t\geq \sqrt{2 \beta \sigma^2\theta^2 }\cdot  t \right)\leq \beta\left(1-\theta^2\right).
\end{equation}
 
%{\bf Case II.} If $\left(\beta, \sigma^2\right) \in \mathcal{C}_\text{\uppercase\expandafter{\romannumeral1}}$ and $\beta > 1,\sigma ^2 \geq 1$, then
%$f_1^{\prime}(u)\geq 0$ for $0<u<1$.  Thus $\max_{0\leq u\leq 1}f_1(u)=f_1(1)$ and hence
%\begin{equation}\label{da2}
%\limsup_{t \rightarrow \infty} \frac{1}{t} \ln \P\left(M_t\geq\sqrt{2 \beta \sigma^2\theta^2 }\cdot  t \right)\leq  \beta\left(1-\theta^2\right).
%\end{equation}

{\bf Case II.} If $\left(\beta, \sigma^2\right) \in \mathcal{C}_\text{\uppercase\expandafter{\romannumeral1}}$ and $\beta >1,\sigma ^2 <  1$, then obviously $u\mapsto f_1^{\prime}(u)$ is decreasing. Then we have three subcases.

First, if $\theta^2<  \frac{(\beta -1)\sigma ^2}{\left(1-\sigma^2\right)  \beta}$, which implies $f_1^{\prime}(1)>0$, then $\max_{0\leq u\leq 1}f_1(u)=f_1(1)$ and 
\begin{equation}\label{da3}
  \limsup_{t \rightarrow \infty} \frac{1}{t} \ln \P\left(M_t\geq\sqrt{2 \beta \sigma^2\theta^2 }\cdot  t \right)\leq  \beta\left(1-\theta^2\right).
\end{equation}

Second, if  $\theta^2> \frac{(\beta-1) }{\left(1-\sigma^2\right) \sigma ^2\beta}$,\, which means $f_1^{\prime}(0)<0$,  then $\max_{0\leq u\leq 1}f_1(u)=f_1(0)$ and 
\begin{equation}\label{da4}
  \limsup_{t \rightarrow \infty} \frac{1}{t} \ln \P\left(M_t\geq\sqrt{2 \beta \sigma^2\theta^2 }\cdot  t \right)\leq  1-\theta^2 \beta \sigma^2.
\end{equation}

Last, if $ \frac{(\beta -1)\sigma ^2}{\left(1-\sigma^2\right)  \beta} \leq \theta^2\leq \frac{(\beta-1) }{\left(1-\sigma^2\right) \sigma ^2\beta}$, %then $f_1^{\prime}(0)\geq 0$ and $f_1^{\prime}(1)\leq 0$. Thus $\max_{0\leq u\leq 1}f_1(u)=f_1(u_0)$, where $f_1^{\prime}(u_0)=0$. 
then  one can check that $f_1^{\prime}(0)\geq 0$ and $f_1^{\prime}(1)\leq 0$. Thus $\max_{0\leq u\leq 1}f_1(u)=f_1(u_0)$ with
\begin{equation}\label{u0}
  u_0=\frac{\sqrt{\left|\frac{\sigma^2-1}{1-\beta}\right|} \cdot \sigma \theta \cdot \sqrt{\beta}-1}{\sigma^2-1}.
\end{equation}
 
 \noindent Therefore,
\begin{equation}\label{da5}
  \limsup_{t \rightarrow \infty} \frac{1}{t} \ln \P\left(M_t\geq\sqrt{2 \beta \sigma^2\theta^2 }\cdot  t \right)\leq f_1(u_0)=  \frac{-2 \sqrt{\beta(\beta-1)\left(1-\sigma^2\right)} \sigma \theta+\beta-\sigma^2}{1-\sigma^2}.
\end{equation}
This proves the upper bound. \\

\begin{rem} Notice that
$$
\operatorname{Var}\left(\sigma B_s+B_t-B_s\right)=\sigma^2 \cdot s+(t-s). 
$$

\noindent Then applying Lemma \ref{lemMany-one} and \eqref{BMasy},  we also have
\begin{align}
    I_2 & \leq \alpha \int_0^t e^{\beta s+(t-s)}  \P\left(\sigma B_s+B_t-B_s \geqslant \sqrt{2 \beta \sigma^2\theta^2 }\cdot  t \right) \d s 
    \cr & \leq \alpha \int_0^t e^{\beta s+(t-s)} \cdot \frac{1}{\sqrt{2 \pi\left(\sigma^2 s+t-s\right)}} \cdot e^{-\frac{2 \beta\sigma^2  \theta^2 t^2}{2\left(\sigma^2 s+t-s\right)}} \d s \cr
    & = \alpha \int_0^1 e^{f_1(u)t} \cdot \frac{\sqrt{t}}{\sqrt{2 \pi\left((\sigma^2 -1)u+1\right)}}  \d u.
\end{align} 

%\noindent Notice that $f_1(1)=\beta(1-\theta^2)$. Then our task is reduced to know $\max_{0\leq u\leq 1}f_1(u)$. 
%\noindent To end this, let us consider
%\begin{equation}\label{f_1'u}
%f_1^{\prime}(u)=\beta-1+\frac{\beta\sigma^2  \theta^2}{\left[\left(\sigma^2-1\right) u+1\right]^2}\left(\sigma^2-1\right).
%\end{equation}

\noindent If $\left(\beta, \sigma^2\right) \in \mathcal{C}_\text{\uppercase\expandafter{\romannumeral1}}$, $\beta \leq 1$ and $\theta^2\leq \frac{(1-\beta)\sigma^2}{(\sigma^2-1)\beta}$, then one can check that $\max_{0\leq u\leq 1}f_1(u)=f_1(u_0)$. So we will get
\begin{equation*}%\label{da1}
   \limsup_{t \rightarrow \infty} \frac{1}{t} \ln \P\left(M_t\geq \sqrt{2 \beta \sigma^2\theta^2 }\cdot  t \right)\leq f'(u_0),
\end{equation*}
which is larger than $\beta(1-\theta^2).$ Thus we would get a rough upper bound. 
\end{rem}
\end{proof}

\subsection{Lower bound for Theorem \ref{ThmcaseI}}\label{SecLowerI}

Notice that
\begin{align}\label{C1L00a}
& \P\left(M_t \geqslant \sqrt{2 \beta  \sigma^2\theta^2} \cdot  t\right) \cr
\geqslant & \max \left\{\P\left(\exists\, u \in \cN_t^1: X_t^u \geqslant \sqrt{2 \beta \sigma^2 \theta^2} \cdot  t\right), \P\left(\exists\, u \in \cN_t^2: X_t^u \geqslant \sqrt{2 \beta \sigma^2 \theta^2} \cdot   t\right)\right\}.
\end{align}

\noindent First, from \eqref{cvgsbbm}, we have
\begin{align}\label{C1L00b}
\lim_{t\rightarrow\infty}\frac{1}{t}\ln \P\left(\exists\, u \in \cN_t^1: X_t^u \geqslant \sqrt{2 \beta \sigma^2 \theta^2} \cdot   t\right)  &=\lim_{t\rightarrow\infty}\frac{1}{t}\ln \P_{\sigma^2, \beta}\left(\bar{M}_t \geqslant \sqrt{2 \beta  \sigma^2 \theta^2 }  \cdot t\right) \cr &= \beta\left(1-\theta^2\right).
\end{align}

\noindent Recall \eqref{C1U02}. Then
\begin{align}\label{C1L02}
    & \P\left(\exists\, u \in \cN_t^2: X_t^u \geqslant \sqrt{2 \beta \sigma^2 \theta^2} \cdot   t\right) \cr
  & \geq 
  1-\E\left[\exp\left\{-\alpha\int_{I_t} \sum_{u\in \cN_s^1(x_t)}  \bar{u}\left(t-s, \sqrt{2 \beta\sigma^2 \theta^2} \cdot  t-x_t \right){\rm d} s\right\}\right]  \cr
  & = 
  1-\E\left[\exp\left\{-\alpha\int_{I_t} \left| \cN_s^1(x_t)\right|  \bar{u}\left(t-s, \sqrt{2 \beta\sigma^2 \theta^2} \cdot  t-x_t \right){\rm d} s\right\}\right],
\end{align}
where $I_t\subset [0, t]$ is an interval of length $1$, $x_t\in {\mathbb R}$ is a deterministic function of $t$, and
\[
\cN_s^1(a)=\{u\in \cN_s^1: X_s^u\geq a\}.
\]
Both $I_t$ and $x_t$ will be determined later from case to case.  Then, applying Jensen's inequality and Fubini's theorem, gives % and some constant $A>0$.

\begin{align}\label{C1L03}
& \P\left(\exists\, u \in \cN_t^2: X_t^u \geqslant \sqrt{2 \beta \sigma^2 \theta^2} \cdot   t\right) \cr
& \geq 
 \int_{I_t} 1- \E\left[\exp\left\{-\alpha\left| \cN_s^1(x_t)\right|   \bar{u}\left(t-s, \sqrt{2 \beta\sigma^2 \theta^2} \cdot  t-x_t \right)\right\}\right] {\rm d} s.
\end{align}

\subsubsection{\bf The case of $\max_{0\leq u\leq 1} f_1(u)=f_1(0)$. }\label{SecCaseI01}

Choose 
\[
I_t=[\varepsilon t, \varepsilon t+1],\quad x_t=-\varepsilon t,
\]
for  $\varepsilon>0$ small enough. Then for $s\in I_t$,
\begin{align*}
&   \E\left[\exp\left\{-\alpha\left| \cN_s^1(x_t)\right|  \bar{u}\left(t-s, \sqrt{2 \beta\sigma^2 \theta^2} \cdot  t-x_t \right)\right\}\right] \cr
& \leq \E\left[\exp\left\{-\alpha   \bar{u}\left(t-s, \sqrt{2 \beta\sigma^2 \theta^2} \cdot  t+\varepsilon t \right)\right\}, |\cN_s^1(-\varepsilon t)  | \geq 1\right] \cr
& \quad + \P\left( |\cN_s^1(-\varepsilon t)  | =0  \right)
\cr
& =\exp\left\{-\alpha  \bar{u}\left(t-s, \sqrt{2 \beta\sigma^2 \theta^2} \cdot  t+\varepsilon t \right)\right\}\P_{\sigma^2, \beta}\left( \bar{M}_{s} \geq  -\varepsilon t\right) + \P_{\sigma^2, \beta}\left( \bar{M}_{s} < -\varepsilon t\right).
\end{align*}

\noindent This, together with \eqref{C1L03}, gives
\begin{align*}
    & \P\left(\exists\, u \in \cN_t^2: X_t^u \geqslant \sqrt{2 \beta \sigma^2 \theta^2} \cdot   t\right) \cr
  &\geq 
\int_{I_t}  \left(1- \exp\left\{-\alpha  \bar{u}\left(t-s, \sqrt{2 \beta\sigma^2 \theta^2} \cdot  t+\varepsilon t \right)\right\}\right)\P_{\sigma^2, \beta}\left( \bar{M}_{s} \geq  -\varepsilon t\right) {\rm d }s
 \cr &\geq 
\int_{0}^1  \left(1- \exp\left\{-\alpha  \bar{u}\left(t-\varepsilon t-s, \sqrt{2 \beta\sigma^2 \theta^2} \cdot  t+\varepsilon t \right)\right\}\right)\P_{\sigma^2, \beta}\left( \bar{M}_{\varepsilon t+s} \geq  -\varepsilon t\right) {\rm d }s.
\end{align*}

\noindent Then applying \eqref{cvgsbbm}, we have 
\[
\lim_{ \varepsilon \rightarrow 0}  \lim_{t \rightarrow \infty}\frac{1}{t}\ln\bar{u}\left(t-\varepsilon t-s, \sqrt{2 \beta\sigma^2 \theta^2} \cdot  t+\varepsilon t \right)\geq  1-\beta \sigma^2 \theta^2,
\]
uniformly for $s\in [0,1]$. We further  notice  that 
 \[
  \P_{\sigma^2, \beta}\left( \bar{M}_{s} \geq -\varepsilon t\right)
  \rightarrow 1,\quad \text{as }t\rightarrow\infty,
\]
uniformly for $s\in [0,1]$. Then applying Jensen's inequality and Fatou's lemma again, one has
\begin{align*}
 \liminf_{t \rightarrow \infty}\frac{1}{t}\ln    \P\left(\exists\, u \in \cN_t^2: X_t^u \geqslant \sqrt{2 \beta \sigma^2 \theta^2} \cdot   t\right) 
  =
  1-\beta\sigma^2  \theta^2
  \geq \beta (1-\theta^2),
\end{align*}
where the last inequality holds if $\theta^2\geq \frac{(\beta-1) }{\left(1-\sigma^2\right) \beta} $.
This proves Theorem \ref{ThmcaseI} for the case of $\theta^2\geq \frac{(\beta-1) }{\left(1-\sigma^2\right) \sigma ^2\beta}$, when $\sigma^2<1$ and $\beta>1$.

%\newpage

\subsubsection{\bf The case of $\max_{0\leq u\leq 1} f_1(u)=f_1(1)$. }\label{SecCaseI02} 

This case corresponds to the case that
\begin{align*}
&\left\{\left(\theta^2,\, \beta,\, \sigma^2\right):\theta^2>1,\, \beta\leq 1,\, \sigma^2>\frac{1}{\beta}\right\}
\bigcup 
\left\{\left(\theta^2,\, \beta,\, \sigma^2\right):\theta^2>1,\, \beta>1,\, \sigma^2>1\right\}\cr
&\quad 
\bigcup
\left\{\left(\theta^2, \beta, \sigma^2\right): \theta^2\leq \frac{(\beta -1)\sigma ^2}{\left(1-\sigma^2\right)  \beta},\, \beta>1,\, \frac{\beta}{2\beta-1}<\sigma^2<1\right\}.
\end{align*}
One can make use \eqref{C1L00a} and \eqref{C1L00b} directly to conclude. We are  done.\qed

%\newpage

\subsubsection{\bf The case of $\max_{0\leq u\leq 1} f_1(u)=f_1(u_0)$. }\label{SecCaseI03}

  This  corresponds to  the case that
$\beta >1, \sigma ^2< 1$, and 

\[\frac{(\beta -1)\sigma ^2}{\left(1-\sigma^2\right)  \beta}\leq \theta^2\leq \frac{(\beta-1) }{\left(1-\sigma^2\right) \sigma ^2\beta}.
\]
Recall $u_0$ from \eqref{u0}. We also recall that
\[
f_1(u_0)=\frac{-2 \sqrt{\beta(\beta-1)\left(1-\sigma^2\right)} \sigma \theta+\beta-\sigma^2}{1-\sigma^2}\geq f_1(1)=\beta(1-\theta^2).\]

\noindent  Choose 
\[
I_t=[u_0 t, u_0 t+1],\quad x_t=x_0 t,
\]
for some $0<x_0< \sqrt{2 \beta\sigma^2 \theta^2}$, whose value will be determined later.  Then for $\varepsilon>0$ small enough and $s\in I_t$,
\begin{align*}
&   \E\left[\exp\left\{-\alpha \left|\cN_s^1(x_t) \right| \bar{u}\left(t-s, \sqrt{2 \beta\sigma^2 \theta^2} \cdot  t-x_t \right)\right\}\right] \cr
& \leq \exp\left\{-\alpha   \bar{u}\left(t-s, \sqrt{2 \beta\sigma^2 \theta^2} \cdot  t-x_0 t \right)\right\}\P\left( |\cN_s^1(x_0t)  | \geq 1 \right)  + \P\left( |\cN_s^1(x_0t)  | =0  \right),
\end{align*}
where 
\[
%\beta_0=  \beta u_0 -\frac{x_0 ^2}{2 \sigma^2 u_0}>0,\quad \text{and}\quad 
x_0=\frac{\sqrt{2\beta } \theta u_0 \sigma^3}{1-u_0+u_0 \sigma ^2}.%=\sqrt{2}u_0\sigma^2\sqrt{\frac{\beta-1}{1-\sigma^2}}.
\]

%\noindent In fact, since in this case $1>\sigma^2\geq \frac{\beta}{2\beta-1}$

\noindent  Therefore,
\begin{align*}
    & \P\left(\exists\, u \in \cN_t^2: X_t^u \geqslant \sqrt{2 \beta \sigma^2 \theta^2} \cdot   t\right) 
 \cr &\geq 
\int_{0}^1  \left(1-\exp\left\{-\alpha  \bar{u}\left(t-u_0t-s, \sqrt{2 \beta\sigma^2 \theta^2} \cdot  t-x_0 t \right)\right\}\right) \cr
&\qquad\quad \times \P\left( |\cN_{u_0t+s}^1(x_0 t)  | \geq 1 \right) {\rm d }s.
\end{align*}

\noindent  Furthermore, the fact that $f_1'(u_0)=0$ gives
\[
1-u_0+u_0\sigma^2=\sqrt{\frac{1-\sigma^2}{\beta-1}} \sqrt{\beta} \sigma \theta.
\]
Then
\[x_0=\sqrt{2}\sigma^2 \sqrt{\frac{\beta-1}{1-\sigma^2}} u_0>\sqrt{2\beta}\sigma u_0\]
and
\[
\sqrt{2\beta\sigma^2 \theta^2}-x_0=\sqrt{2\beta}\sigma\theta \frac{1-u_0}{1-u_0+u_0\sigma^2}=\sqrt{2}\sqrt{\frac{\beta-1}{1-\sigma^2}}(1-u_0)>\sqrt{2}(1-u_0),
\]
because of 
\[
\sigma^2\geq \frac{\beta}{2\beta-1}>2-\beta,\quad \text{for } \beta>1.
\]
Therefore,
\[\frac{\left(\sqrt{2\beta \sigma ^2\theta^2} -x_0\right)^2}{2\left(1-u_0\right)^2}>1,\quad \text{for }\beta>1.
\]

\noindent Thus applying \eqref{cvgsbbm} again, we have 
\begin{align*}
&  \lim_{t \rightarrow \infty}\frac{1}{t}\ln\bar{u}\left(t-u_0 t-s, \sqrt{2 \beta\sigma^2 \theta^2} \cdot  t-x_0 t \right)
\cr & 
\quad \geq  
\left(1-\frac{\left(\sqrt{2\beta \sigma ^2\theta^2} -x_0\right)^2}{2\left(1-u_0\right)^2}\right)\left(1-u_0\right),
\end{align*}
uniform for $s\in [0,1]$ and 
\begin{align*}
 \lim_{t \rightarrow \infty}\frac{1}{t}\ln \P\left( |\cN_{u_0t+s}^1(x_0 t)  | \geq 1 \right)
&= \lim_{t \rightarrow \infty}\frac{1}{t}\ln \P_{\sigma^2, \beta}\left(\bar{M}_{u_0 t+s}>x_0 t\right) \cr&=
\left(1-\frac{x_0^2}{{2\beta}\sigma^2 u_0^2}\right)\beta u_0
 = \beta u_0 -\frac{x_0 ^2}{2 \sigma^2 u_0},
\end{align*}
uniform for $s\in [0,1]$. Then applying Jensen's inequality and Fatou's lemma, one has
\begin{align}\label{xiao1}
& \liminf_{t \rightarrow \infty}\frac{1}{t}\ln    \P\left(\exists\, u \in \cN_t^2: X_t^u \geqslant \sqrt{2 \beta \sigma^2 \theta^2} \cdot   t\right) 
 \cr & \geq 
 \beta u_0 -\frac{x_0 ^2}{2 \sigma^2 u_0}+\left(1-\frac{\left(\sqrt{2\beta \sigma ^2\theta^2} -x_0\right)^2}{2\left(1-u_0\right)^2}\right)\left(1-u_0\right)
\cr&  = (\beta-1) u_0+1-\frac{\sigma^2 \beta \theta^2}{(\left.\sigma^2-1\right) u_0+1}
\cr&  = 
\frac{-2 \sqrt{\beta(\beta-1)\left(1-\sigma^2\right)} \sigma \theta+\beta-\sigma^2}{1-\sigma^2}.
\end{align}

\noindent This proves Theorem \ref{ThmcaseI} for the case  of $\max_{0\leq u\leq 1} f_1(u)=f_1(u_0)$. We are done.

\section{Proof of Theorem \ref{ThmcaseII}}\label{SecCaseII}
In this section, we always suppose that
\[(\beta, \sigma^2)\in\mathcal{C}_{\text{\uppercase\expandafter{\romannumeral2}}}  =\left\{\left(\beta, \sigma^2\right): \sigma^2<\frac{\mathbb{1}_{\{\beta \leq 1\}}}{\beta}+\mathbb{1}_{\{\beta>1\}}(2-\beta)\right\}.
\]
One can use arguments similar to those in Section \ref{SecCaseI}.  We just replace $\sqrt{2\beta \sigma^2}\theta$ by $\sqrt{2} \theta$.  However, for the upper bound we have a simpler method.

\subsection{Upper bound for Theorem \ref{ThmcaseII}}
We first consider the upper bound.
\begin{proof}
First,
$$
\begin{aligned}
& \P\left(M_t \geqslant \sqrt{2 } \theta t\right) \cr
& \leqslant 2\max \left[ e^{\beta t}  \P\left(\sigma B_t \geqslant \sqrt{2 } \theta t\right),\alpha \int_0^t e^{\beta s+(t-s)}  \P\left(\sigma B_s+B_t-B_s \geqslant \sqrt{2 } \theta t\right) d s\right]\cr
&=:2\max ( \text{\uppercase\expandafter{\romannumeral1}},\,\text{\uppercase\expandafter{\romannumeral2}}).
\end{aligned}
$$

\noindent By \eqref{BMasy}, we have 
$$
\begin{aligned}
  \text{\uppercase\expandafter{\romannumeral1}}  \leq O(1) \cdot \frac{1}{\sqrt{t}} \cdot e^{\beta t} \cdot e^{-\frac{2  \theta^2 t^2}{2\sigma ^2 t}} =\frac{O(1)}{\sqrt{t}} \cdot e^{\left(\beta -\frac{\theta^2}{\sigma ^2} \right) t}.
\end{aligned}
$$
and  by Lemma \ref{lemMany-one},  we also have
\begin{align}
   \text{\uppercase\expandafter{\romannumeral2}} & \leq \alpha \int_0^t e^{\beta s+(t-s)}  \P\left(\sigma B_s+B_t-B_s \geqslant \sqrt{2 \theta^2 }\cdot  t \right) \d s 
    \cr & \leq \alpha \int_0^t e^{\beta s+(t-s)} \cdot \frac{1}{\sqrt{2 \pi\left(\sigma^2 s+t-s\right)}} \cdot e^{-\frac{2  \theta^2 t^2}{2\left(\sigma^2 s+t-s\right)}} \d s \cr
    & = \alpha \int_0^1 e^{f_2(u)t} \cdot \frac{\sqrt{t}}{\sqrt{2 \pi\left((\sigma^2 -1)u+1\right)}}  \d u,
\end{align} 
%\[
%text{\uppercase\expandafter{\romannumeral2}} \leq \alpha \int_0^t e^{\beta s+(t-s)} \cdot \frac{1}{\sqrt{2 \pi\left(\sigma^2 s+t-s\right)}} \cdot e^{-\frac{2 \theta^2 t^2}{2\left(\sigma^2 s+t-%s\right)}}{\rm d}s. \]
where
\begin{equation}\label{2fu}
  f_2(u)=(\beta-1) u+1-\frac{\theta^2}{(\left.\sigma^2-1\right) u+1}, \quad u \in(0,1).
\end{equation}

%\noindent and 
%$$
%f_2(u_0)=\frac{2 \sqrt{(\beta-1)\left(1-\sigma^2\right)}  \theta+\beta-\sigma^2}{1-\sigma^2}.
%$$

\noindent Then we have several cases.

{\bf Case I.} $\left(\beta, \sigma^2\right) \in \mathcal{C}_\text{\uppercase\expandafter{\romannumeral2}}$ and $\beta >1$.

\noindent In this case $u\mapsto f_2^{\prime}(u)$ is increasing and $f_2^{\prime}(1)< 0$. Thus $\max_{0\leq u\leq 1}f_2(u)=f_2(0)$ and hence
\begin{equation}\label{2da1}
   \lim _{t \rightarrow \infty} \frac{1}{t} \ln \P\left(M_t\geq \sqrt{2 } \theta t\right)\leq  1-\theta^2.
\end{equation}

{\bf Case II.} $\left(\beta, \sigma^2\right) \in \mathcal{C}_\text{\uppercase\expandafter{\romannumeral2}}$ and $\beta \leq 1,\sigma ^2 < 1$.

\noindent In this case $f_2^{\prime}(u)< 0$ for all $u\in [0,1]$, Thus $\max_{0\leq u\leq 1}f_2(u)=f_2(0)$ and hence

\begin{equation}\label{2da2}
\lim _{t \rightarrow \infty} \frac{1}{t} \ln \P\left(M_t\geq\sqrt{2 } \theta t\right)\leq  1-\theta^2.
\end{equation}

{\bf Case III.} $\left(\beta, \sigma^2\right) \in \mathcal{C}_\text{\uppercase\expandafter{\romannumeral2}}$ and $\beta \leq 1,\sigma ^2 \geq   1$. 
In this case, $u\mapsto f_2^{\prime}(u)$ is decreasing.

When  $\theta^2<\frac{(\beta -1) }{\left(1-\sigma^2\right) }$, then $f_2^{\prime}(0)< 0$. Thus $\max_{0\leq u\leq 1}f_2(u)=f_2(0)$ and hence
\begin{equation}\label{2da3}
  \lim _{t \rightarrow \infty} \frac{1}{t} \ln \P\left(M_t\geq\sqrt{2 } \theta t\right)\leq  1- \theta^2.
\end{equation}

When $\theta^2>\frac{(1-\beta) }{\left(\sigma^2-1\right) }\sigma ^4$, then $f_2^{\prime}(1)>0$, Thus $\max_{0\leq u\leq 1}f_2(u)=f_2(1)$ and hence
\begin{equation}\label{2da4}
  \lim _{t \rightarrow \infty} \frac{1}{t} \ln \P\left(M_t\geq\sqrt{2 } \theta t\right)\leq  \beta -\frac{\theta ^2}{\sigma ^2}.
\end{equation}

When $\frac{(\beta -1) }{\left(1-\sigma^2\right) } \leq   \theta^2 \leq   \frac{(1-\beta) }{\left(\sigma^2-1\right) }\sigma ^4$, then $f_2^{\prime}(0)\geq 0$ and $f_2^{\prime}(1)\leq 0$. Thus $\max_{0\leq u\leq 1}f_2(u)=f_2(u_0)$ with $f_2^{\prime}(u_0)=0$. Then we have
\begin{equation}\label{2u0}
  u_0=\frac{\sqrt{\left|\frac{\sigma^2-1}{1-\beta}\right|} \cdot  \theta  -1}{\sigma^2-1}
\end{equation} 
and hence
\begin{equation}\label{2da5}
  \lim _{t \rightarrow \infty} \frac{1}{t} \ln \P\left(M_t\geq\sqrt{2 } \theta t\right)\leq f_2(u_0)=  \frac{2 \sqrt{(\beta-1)\left(1-\sigma^2\right)}  \theta+\beta-\sigma^2}{1-\sigma^2}.
\end{equation}

\end{proof}

\subsection{Lower bound for Theorem \ref{ThmcaseII}}

Notice that
\begin{align}\label{C2L00a}
& \P\left(M_t \geqslant \sqrt{2} \theta  t\right) \cr
&\geqslant  \max \left\{\P\left(\exists\, u \in \cN_t^1: X_t^u \geqslant\sqrt{2} \theta  t\right), \P\left(\exists\, u \in \cN_t^2: X_t^u \geqslant \sqrt{2} \theta   t\right)\right\}.
\end{align}

\noindent First, from \eqref{cvgsbbm}, we have
\begin{align}\label{C2L00b}
\lim_{t\rightarrow\infty}\frac{1}{t}\ln \P\left(\exists\, u \in \cN_t^1: X_t^u \geqslant \sqrt{2 } \theta   t \right) 
 &=\lim_{t\rightarrow\infty}\frac{1}{t}\ln \P_{\sigma^2, \beta}\left(\bar{M}_t \geqslant \sqrt{2 }  \theta t\right) \cr & = \beta-\frac{\theta^2}{\sigma^2}.
\end{align}

\noindent By similar reasoning as \eqref{C1L02} and \eqref{C1L03}, we have
\begin{align*}
    & \P\left(\exists\, u \in \cN_t^2: X_t^u \geqslant \sqrt{2} \theta   t\right)
    \cr &\geq 
   \int_{I_t}1-   \E\left[\exp\left\{-\alpha\left|\cN_s^1(x_t)\right|  \bar{u}\left(t-s, \sqrt{2 }\theta t-x_t \right)\right\}\right] {\rm d} s,
\end{align*}
where $I_t\subset [0, t]$ is an interval of length $1$, $x_t\in {\mathbb R}$ is a deterministic function of $t$. 
Both $I_t$ and $x_t$ will be determined later from case to case. Recall from \eqref{2fu} that
 \[ f_2(u)=(\beta-1) u+1-\frac{\theta^2}{(\left.\sigma^2-1\right) u+1}, \quad u \in(0,1).
\]

\subsubsection{\bf The case of $\max_{0\leq u\leq 1} f_2(u)=f_2(0)$. }\label{SecCaseII01}

This corresponds to the cases that 
 $\beta >1$, or $\beta \leq 1, \sigma ^2 \leq   1$, or $\beta \leq 1, \sigma ^2 > 1$ and $ \theta^2<\frac{(\beta -1) }{\left(1-\sigma^2\right) }.$   Choose 
\[
I_t=[\varepsilon t, \varepsilon t+1],\quad x_t=-\varepsilon t,
\]
for  $\varepsilon>0$ small enough. Then 
\begin{align*}
    & \P\left(\exists\, u \in \cN_t^2: X_t^u \geqslant \sqrt{2 } \theta   t\right) \cr
  &\geq 
  \int_{0}^1  \left(1- \exp\left\{-\alpha  \bar{u}\left(t-\varepsilon t-s,  \sqrt{2 } \theta    t+\varepsilon t \right)\right\}\right)\P_{\sigma^2, \beta}\left( \bar{M}_{\varepsilon t+s} \geq  -\varepsilon t\right) {\rm d }s.
\end{align*}

\noindent Then applying \eqref{cvgsbbm}, we have 
\[
\lim_{ \varepsilon \rightarrow 0}  \liminf_{t \rightarrow \infty}\frac{1}{t}\ln\bar{u}\left(t-s,  \sqrt{2 } \theta   t+\varepsilon t \right)\geq  1-  \theta^2.
\]

\noindent uniformly for $s\in[0,1]$. We further  notice  that 
 \[
  \P_{\sigma^2, \beta}\left( \bar{M}_{\varepsilon t+s} \geq -\varepsilon t\right)
  \rightarrow 1,\quad \text{as }t\rightarrow\infty,
\]
\noindent  uniformly for $s\in[0,1]$. Then applying Jensen's inequality and Fatou's lemma, one has
\begin{align*}
  \liminf_{t \rightarrow \infty}\frac{1}{t}\ln    \P\left(\exists\, u \in \cN_t^2: X_t^u \geqslant  \sqrt{2 } \theta      t\right) 
  & \geq 
  1- \theta^2.
\end{align*}
This proves Theorem \ref{ThmcaseII} for the cases that  $\beta >1$, or $\beta \leq 1, \sigma ^2 \leq   1$, or $\beta \leq 1, \sigma ^2 > 1$ and $ \theta^2<\frac{(\beta -1) }{\left(1-\sigma^2\right) }$, since  $1- \theta^2\geq \beta-\frac{\theta^2}{\sigma^2}$ in all above cases.

\subsubsection{\bf The case of $\max_{0\leq u\leq 1} f_2(u)=f_2(1)$. }\label{SecCaseII02} 

This case corresponds to the case that
$\beta \leq 1, 1<\sigma ^2 \leq \frac{1}{\beta}$ and $ \theta^2\geq \frac{(1-\beta) }{\left(\sigma^2-1\right) }\sigma ^4.$  Then with \eqref{C2L00a} and \eqref{C2L00b} in hand,
 Theorem \ref{ThmcaseII} for the case that
$\beta \leq 1, 1<\sigma ^2 \leq \frac{1}{\beta}$ and $ \theta^2\geq \frac{(1-\beta) }{\left(\sigma^2-1\right) }\sigma ^4$ follows readily. 

\subsubsection{\bf The case of $\max_{0\leq u\leq 1} f_2(u)=f_2(u_0)$ for $0<u_0<1$. }\label{SecCaseII03}

Notice that  $f_2'(u_0)=0$.  This  corresponds to  the case that
$\beta <1,  1<\sigma ^2 < \frac{1}{\beta}$, and 
\[
\frac{(1-\beta) }{\left(\sigma^2-1\right) } <   \theta^2 <  \frac{(1-\beta) }{\left(\sigma^2-1\right) }\sigma ^4.
\]
We also recall that in this case
\[0<u_0=\frac{\sqrt{\frac{\sigma^2-1}{1-\beta}} \cdot  \theta  -1}{\sigma^2-1}<1\]
and
\[
f_2(u_0)=\frac{-2 \sqrt{(1-\beta)\left(\sigma^2-1\right)} \theta-\beta+\sigma^2}{\sigma^2-1}.\]

\noindent Choose 
\[
%I_t=[u_0 t, u_0 t+1],\quad 
x_0=\frac{\sqrt{2 } \theta u_0 \sigma^2}{1-u_0+u_0 \sigma ^2},\quad x_t=x_0 t.
\]

\noindent One can check that $x_0<\sqrt{2}\theta$ since $u_0>0.$
\noindent  Then,
\begin{align*}
    & \P\left(\exists\, u \in \cN_t^2: X_t^u \geqslant \sqrt{2} \theta   t\right) 
 \cr &\geq 
\int_{0}^1  \left(1-e^{-\alpha   \bar{u}\left(t-u_0t-s, \sqrt{2} \theta  t-x_0 t \right)}\right)  \P\left( |\cN_{u_0t+s}^1(x_t)  | \geq 1 \right) {\rm d }s.
\end{align*}

\noindent We also notice that
\[\frac{\left(\sqrt{2} \theta-x_0\right)^2}{2\left(1-u_0\right)^2}>1,\]
since
\[
\sigma^2<\frac{1}{\beta}<2-\beta,\quad  \text{for}\quad  \beta<1.
\]
 
 \noindent Then applying \eqref{cvgsbbm} again, we have 
\begin{align*}
 \liminf_{t \rightarrow \infty}\frac{1}{t}\ln\bar{u}\left(t-u_0 t-s, \sqrt{2 }\theta t-x_0 t \right)
 \geq  
\left(1-\frac{\left(\sqrt{2} \theta-x_0\right)^2}{2\left(1-u_0\right)^2}\right)\left(1-u_0\right),
\end{align*}
uniform for $s\in [0,1]$  and 
\[
 \lim_{t\rightarrow\infty}\frac{1}{t}\ln\P\left( |\cN_{u_0 t+s}^1(x_t)  | \geq 1 \right) = \beta u_0 -\frac{x_0 ^2}{2 \sigma^2 u_0},
\]
uniform for $s\in [0,1]$. Then one has
\begin{align}\label{xiao1a}
&   \liminf_{t \rightarrow \infty}\frac{1}{t}\ln    \P\left(\exists\, u \in \cN_t^2: X_t^u \geqslant \sqrt{2}\theta   t\right) 
 \cr & \geq 
\beta u_0 -\frac{x_0 ^2}{2 \sigma^2 u_0}+\left(1-\frac{\left(\sqrt{2} \theta-x_0\right)^2}{2\left(1-u_0\right)^2}\right)\left(1-u_0\right)
\cr&  = (\beta-1) u_0+1-\frac{ \theta^2}{\left(\sigma^2-1\right) u_0+1}
\cr&  = 
\frac{-2 \sqrt{(1-\beta)\left(\sigma^2-1\right)} \theta-\beta+\sigma^2}{\sigma^2-1}.
\end{align}

\noindent This proves Theorem \ref{ThmcaseII} for the above case.

\section{Proof of Theorem \ref{ThmcaseIII}}\label{SecCaseIII}

In this section, we always assume that
$$
(\beta,\sigma^2)\in \mathcal{C}_{\text{\uppercase\expandafter{\romannumeral3}}}  =\left\{\left(\beta, \sigma^2\right): \sigma^2+\beta>2 \text { and } \sigma^2<\frac{\beta}{2 \beta-1}\right\}.
$$
Recall $\xi =\frac{\sigma ^2-\beta }{\sqrt{2(1-\sigma ^2)(\beta -1)} }. $
The arguments are also exactly the same as those in Section \ref{SecCaseII}. One just replace $\theta$ by $\frac{\xi \theta}{\sqrt{2}}$. We give only an outline here. 

\subsection{Upper bound}
\begin{proof}
First, by union bound, one has
\begin{align*}
& \P\left(M_t \geqslant \xi    \theta t\right) \cr
\leqslant &2\max \left( e^{\beta t}  \P\left(\sigma B_t \geqslant \xi    \theta t\right),\alpha \int_0^t e^{\beta s+(t-s)}  \P\left(\sigma B_s+B_t-B_s \geqslant \xi    \theta t\right) {\rm d} s\right)\cr
=: & 2\max\left(
\text{\uppercase\expandafter{\romannumeral1}},\,\text{\uppercase\expandafter{\romannumeral2}}\right).
\end{align*}

\noindent According to \eqref{BMasy}, we have
\[
\text{\uppercase\expandafter{\romannumeral1}} \leq O(1) \cdot \frac{1}{\sqrt{t}} \cdot e^{\beta t} \cdot e^{-\frac{ \xi ^2\theta^2 t^2}{2 \sigma ^2t}} =\frac{O(1)}{\sqrt{t}} \cdot e^{\left(\beta-\frac{ \xi ^2\theta^2 }{2 \sigma^2}\right) t}
%\frac{\theta^2\left(\sigma^2-\beta\right)^2}{4 \sigma^2\left(1-\sigma^2\right)(\beta-1)}\right) t}
\]
and
$$
\text{\uppercase\expandafter{\romannumeral2}} \leq \alpha \int_0^t e^{\beta s+(t-s)} \cdot \frac{1}{\sqrt{2 \pi\left(\sigma^2 s+t-s\right)}} \cdot e^{-\frac{ \xi ^2 \theta^2 t^2}{2\left(\sigma^2 s+t-s\right)}}{\rm d}s. \\
$$

\noindent For \text{\uppercase\expandafter{\romannumeral2}}, set
\begin{equation}\label{3fu}
  f_3(u)=(\beta-1) u+1-\frac{\xi^2 \theta^2}{2((\left.\sigma^2-1\right) u+1)}, \quad u \in(0,1).
\end{equation}
\noindent  Then
$$
f_3^{\prime}(u)=\beta-1+\frac{\xi^2 \left(\sigma^2-1\right)\theta^2}{2\left[\left(\sigma^2-1\right) u+1\right]^2}.
$$

\noindent and
$$
f_3(u_0)=\frac{\sqrt{2}(1-\beta)-\theta^2\left(\sigma^2-\beta\right)^2-1}{\sqrt{2}\left(\sigma^2-1\right)}+1.
$$

\noindent 
Since $\left(\beta, \sigma^2\right) \in \mathcal{C}_\text{\uppercase\expandafter{\romannumeral3}}$, then $\beta >1,\sigma ^2 < 1$ and hence $u\mapsto f_3^{\prime}(u)$ is decreasing. We have three cases. 

\noindent If $\theta^2\geq  \frac{2(\beta-1)}{1-\sigma^2}\frac{1}{\xi^2}$, then $f_3^{\prime}(0)\leq  0$. Thus $\max_{0\leq u\leq 1}f_3(u)=f_3(0)$ and hence
\begin{equation}\label{3da3}
  \limsup_{t \rightarrow \infty} \frac{1}{t} \ln \P\left(M_t\geq\xi  \theta t\right)\leq f_3(0)=  1-\frac{\xi^2\theta^2}{2}.
\end{equation}

\noindent  If $\theta^2\leq \frac{2(\beta-1)}{1-\sigma^2}\frac{\sigma^4}{\xi^2}$, then $f_3^{\prime}(1)\geq 0$. Thus $\max_{0\leq u\leq 1}f_3(u)=f_3(1)$ and hence
\begin{equation}\label{3da4}
  \limsup_{t \rightarrow \infty} \frac{1}{t} \ln \P\left(M_t\geq\xi  \theta t\right)\leq f_3(1)=   \beta-\frac{\xi^2 \theta^2}{2 \sigma^2}.
\end{equation}

\noindent  If $\frac{2(\beta-1)}{1-\sigma^2}\frac{\sigma^4}{\xi^2} < \theta^2< \frac{2(\beta-1)}{1-\sigma^2}\frac{1}{\xi^2}$, then $f_3^{\prime}(0)>0$  and $f_3^{\prime}(1)<0$. Thus $\max_{0\leq u\leq 1}f_3(u)=f_3(u_0)$ with $f_3^{\prime}(u_0)=0$ yields
that
\begin{equation}\label{3u0}
  u_0=\frac{ \xi\theta \sqrt{\frac{1-\sigma^2}{2(\beta-1)}} -1}{\sigma^2-1}
\end{equation}
and
\begin{equation}\label{3da5}
  \limsup_{t \rightarrow \infty} \frac{1}{t} \ln \P\left(M_t\geq\xi  \theta t\right)\leq f_3(u_0)=  \frac{\sqrt{2}(1-\beta)-\theta^2\left(\sigma^2-\beta\right)^2-1}{\sqrt{2}\left(\sigma^2-1\right)}+1.
\end{equation}

\end{proof}

\subsection{Lower bound for Theorem \ref{ThmcaseIII}}

Notice that
\begin{align}\label{C3L00a}
 \P\left(M_t \geqslant \xi \theta  t\right) 
\geqslant  \max \left\{\P\left(\exists\, u \in \cN_t^1: X_t^u \geqslant\xi \theta  t\right),\, \P\left(\exists\, u \in \cN_t^2: X_t^u \geqslant \xi \theta  t\right)\right\}.
\end{align}

\noindent First, from \eqref{cvgsbbm}, we have
\begin{align}\label{C3L00b}
\lim_{t\rightarrow\infty}\frac{1}{t}\ln \P\left(\exists\, u \in \cN_t^1: X_t^u \geqslant \xi \theta   t \right) 
 &=\lim_{t\rightarrow\infty}\frac{1}{t}\ln \P_{\sigma^2, \beta}\left(\bar{M}_t \geqslant \xi \theta  t\right) \cr & = \beta-\frac{\xi^2\theta^2}{2\sigma^2}.
\end{align}

\noindent By similar reasoning as \eqref{C1L02} and \eqref{C1L03}, we have
\begin{align*}
    & \P\left(\exists\, u \in \cN_t^2: X_t^u \geqslant \xi \theta   t\right)
    \cr &\geq 
   \int_{I_t}1-   \E\left[\exp\left\{-\alpha\sum_{u\in \cN_s^1(x_t)}  \bar{u}\left(t-s, \xi \theta  t-x_t \right)\right\}\right] {\rm d} s,
\end{align*}
where $I_t\subset [0, t]$ is an interval of length $1$, $x_t\in {\mathbb R}$ is a deterministic function of $t$. 
Both $I_t$ and $x_t$ will be determined later from case to case. Recall from \eqref{3fu} that
\begin{equation*}
  f_3(u)=(\beta-1) u+1-\frac{\xi^2 \theta^2}{2((\left.\sigma^2-1\right) u+1)}, \quad u \in(0,1).
\end{equation*}

\subsubsection{\bf The case of $\max_{0\leq u\leq 1} f_3(u)=f_3(0)$. }\label{SecCaseIII01}

This corresponds to the cases that 
$\theta^2\geq  \frac{2(\beta-1)}{1-\sigma^2}\frac{1}{\xi^2}$.  Choose 
\[
I_t=[\varepsilon t, \varepsilon t+1],\quad x_t=-\varepsilon t,
\]
for  $\varepsilon>0$ small enough. Then 
\begin{align*}
    & \P\left(\exists\, u \in \cN_t^2: X_t^u \geqslant \xi \theta   t\right) \cr
  &\geq 
  \int_{0}^1  \left(1- \exp\left\{-\alpha  \bar{u}\left(t-\varepsilon t-s,  \xi \theta    t+\varepsilon t \right)\right\}\right)\P_{\sigma^2, \beta}\left( \bar{M}_{\varepsilon t+s} \geq  -\varepsilon t\right) {\rm d }s.
\end{align*}

\noindent Then applying \eqref{cvgsbbm}, we have 
\[
\lim_{ \varepsilon \rightarrow 0}  \liminf_{t \rightarrow \infty}\frac{1}{t}\ln\bar{u}\left(t-s,  \xi \theta   t+\varepsilon t \right)\geq  1-  \frac{\xi^2\theta^2}{2}.
\]

\noindent We further  notice  that 
 \[
  \P_{\sigma^2, \beta}\left( \bar{M}_{\varepsilon t+s} \geq -\varepsilon t\right)%=  P\left( M_{s} \geq -\varepsilon \sqrt{\beta \sigma^2} t\right)
  \rightarrow 1,\quad \text{as }t\rightarrow\infty.
\]
\noindent Then applying Jensen's inequality and Fatou's lemma, one has
\begin{align*} \liminf_{t \rightarrow \infty}\frac{1}{t}\ln    \P\left(\exists\, u \in \cN_t^2: X_t^u \geqslant  \sqrt{2 } \theta      t\right) 
  & \geq 
 1-  \frac{\xi^2\theta^2}{2}.
\end{align*}
This proves Theorem \ref{ThmcaseIII} for the cases  $\theta^2\geq  \frac{2(\beta-1)}{1-\sigma^2}\frac{1}{\xi^2}$.

\subsubsection{\bf The case of $\max_{0\leq u\leq 1} f_3(u)=f_3(1)$. }\label{SecCaseIII02} 

This case corresponds to the case
 $\theta^2\leq  \frac{2(\beta-1)}{1-\sigma^2}\frac{\sigma^4}{\xi^2}$. Similarly as before, \eqref{C3L00a} and \eqref{C3L00b} yield
 Theorem \ref{ThmcaseIII} for this case.

\subsubsection{\bf The case of $\max_{0\leq u\leq 1} f_3(u)=f_3(u_0)$ for $0<u_0<1$. }\label{SecCaseIII03}

Notice that $f_3'(u_0)=0$.  This corresponds to the case 
 $ \frac{2(\beta-1)}{1-\sigma^2}\frac{\sigma^4}{\xi^2} <\theta^2<  \frac{2(\beta-1)}{1-\sigma^2}\frac{1}{\xi^2}$. Choose 
\[
%I_t=[u_0 t, u_0 t+1],\quad 
x_0=\frac{\xi \theta u_0 \sigma^2}{1-u_0+u_0 \sigma ^2},\quad x_t=x_0 t.
\]

\noindent Obviously $x_0<\xi\theta$, since $u_0>0.$
\noindent  Then
\begin{align*}
    & \P\left(\exists\, u \in \cN_t^2: X_t^u \geqslant \xi \theta   t\right) 
 \cr &\geq 
\int_{0}^1  \left(1-e^{-\alpha   \bar{u}\left(t-u_0t-s,\, \xi\theta  t-x_0 t \right)}\right)  \P\left( |\cN_{u_0t+s}^1(x_t)  | \geq 1 \right) {\rm d }s.
\end{align*}

\noindent We also notice that
\[\frac{\left(\xi \theta-x_0\right)^2}{2\left(1-u_0\right)^2}>1.\]

 \noindent Then applying \eqref{cvgsbbm} again, we have 
\begin{align*}
  \liminf_{t \rightarrow \infty}\frac{1}{t}\ln\bar{u}\left(t-u_0 t-s, \xi\theta t-x_0 t \right)
 \geq  
\left(1-\frac{\left(\xi \theta-x_0\right)^2}{2\left(1-u_0\right)^2}\right)\left(1-u_0\right),
\end{align*}
uniform for $s\in [0,1]$ and 
\[
 \lim_{t\rightarrow\infty}\P\left( |\cN_{u_0 t+s}^1(x_t)  | \geq 1  \right) = \beta u_0 -\frac{x_0 ^2}{2 \sigma^2 u_0},
\]
uniform for $s\in [0,1]$. Then one has
\begin{align}\label{xiao1b}
 \liminf_{t \rightarrow \infty}\frac{1}{t}\ln    \P\left(\exists\, u \in \cN_t^2: X_t^u \geqslant \xi\theta   t\right) 
 = 
f_3(u_0).
\end{align}
 This proves Theorem \ref{ThmcaseIII} for the above case. We are done.

\subsubsection*{Acknowledgement}
We are deeply indebted to Haojun Zhao who pointed out an error in the proof of Theorem 1.1 in the previous version.
    %We thank the two referees for their precious work and whose comments allowed to considerably improve the presentation of the results.

\bibliographystyle{abbrv}
\bibliography{bibliobbm.bib}

\bigskip\bigskip\bigskip

\end{document}